\documentclass[12pt,a4paper]{amsart}
\usepackage[active]{srcltx}
\usepackage{enumerate}
\usepackage[parfill]{parskip}
\usepackage[all,2cell]{xy}

\usepackage{amsmath,amssymb,xspace,amsthm}
\newtheorem{theorem}{Theorem}
\newtheorem{example}[theorem]{Example}
\newtheorem{lemma}[theorem]{Lemma}
\newtheorem{proposition}[theorem]{Proposition}
\newtheorem{corollary}[theorem]{Corollary}
\numberwithin{equation}{section}

\renewcommand{\phi}{\varphi}

\renewcommand{\leq}{\leqslant}
\renewcommand{\geq}{\geqslant}

\begin{document}
\title[Simple weight modules over wGWA]{Simple weight modules over\\ weak Generalized Weyl algebras}
\author{Rencai L{\"u}, Volodymyr Mazorchuk and Kaiming Zhao}
\date{\today\\ {\bf Keywords:}  algebra; category; module; weight module\\
{\bf 2010  Math. Subj. Class.:} }

\begin{abstract}
In this paper we address the problem of classification of simple weight modules
over weak generalized Weyl algebras of rank one. The principal difference 
between weak generalized Weyl algebras and generalized weight algebras is that
weak generalized Weyl algebras are defined using an endomorphism rather than an
automorphism of a commutative ring $R$. We reduce classification of simple weight modules
over weak generalized Weyl algebras to description of the dynamics of the action
of the above mentioned endomorphism on the set of maximal ideals. We also describe
applications of our results to the study of generalized Heisenberg algebras.
\end{abstract}
\maketitle

\section{Introduction and description of the results}\label{s1}

Let $R$ be a commutative unital ring, $\sigma:R\to R$ an automorphism and $t\in R$. With the datum $(R,\sigma,t)$
one associates a {\em generalized Weyl algebra} $A=A(R,\sigma,t)$ of rank one defined as an $R$-algebra generated
by elements $X$ and $Y$ subject to the relations
\begin{equation}\label{eq1}
YX=t,\,\, XY=\sigma(t),\,\, Xr=\sigma(r)X,\,\, rY=Y\sigma(r)\,\,\text{for all }r\in R. 
\end{equation}
This class of algebras was introduced by Bavula in \cite{Ba1} and later investigated by various authors,
see, in particular, \cite{Ba2,BJ,DGO,Maz,Sh,Ha2} and references therein. Generalized Weyl algebras were further
studied and generalized to natural higher rank analogues, see for example \cite{BB,MT2,BO}, and then to certain 
twisted and multi parameter versions, see for example \cite{MT1,MT3,Ha1,FH1,FH2} and references therein. Many more
papers studying generalized Weyl algebras can be found using Google search or searching MathSciNet.

One of the main reasons why these algebras attracted such a considerable attention is that the class of generalized
Weyl algebras contains many interesting algebras. Thus the first Weyl algebra, the universal enveloping algebra of
$\mathfrak{sl}_2$, the quantum algebra $U_q(\mathfrak{sl}_2)$ and most of the down-up algebras from \cite{BR} 
are all examples of generalized Weyl algebras of rank one. Higher rank, twisted and multi parameter generalizations
cover much more, we refer the reader to the corresponding papers for examples.

One common feature of all studies mentioned above is that the datum defining a generalized Weyl algebra 
(or its generalization) consists
of a base ring (usually commutative and unital), a collection of automorphisms of this ring, a collection of
elements of this ring and, maybe, some more stuff. At the same time, the relations \eqref{eq1} make perfect sense
for any {\bf endomorphism} $\sigma$ of $R$ (which we always assume to be unital, that is $\sigma(1)=1$). 
To distinguish the latter 
case from the original generalized Weyl algebras we will add the word {\em weak} to the classical name of the 
algebra and abbreviate weak generalized Weyl algebras as {\em wGWA}.

Weakening the requirement for $\sigma$ from being an automorphism to  being an endomorphism leads to a new family 
of interesting examples, namely the so-called generalized Heisenberg algebras, see \cite{CRM,LZ}, defined as 
follows: Let $f(h)\in\mathbb{C}[h]$ be a fixed polynomial. Then the corresponding {\em generalized Heisenberg 
algebra} $\mathcal{H}(f)$ is defined as the $\mathbb{C}$-algebra generated by $X,Y$ and $h$ subject to the 
following relations:
\begin{equation}\label{eq2}
hY = Yf(h);\quad Xh = f(h)X;\quad [X, Y] = f(h)-h.
\end{equation}
These algebras have many applications in theoretical physics, see the references in \cite{LZ} for details.
To realize $\mathcal{H}(f)$ as a wGWA, we may, for example, take 
$R=\mathbb{C}[h,z]$, set $t=h+z$ and choose $\sigma$ as the endomorphism
of $R$ sending $z$ to $z$ and $h$ to $f(h)$. It is immediate to check that, using the correspondence
$X\leftrightarrow X$, $Y\leftrightarrow Y$ and $YX\leftrightarrow h+z$, the defining
relations \eqref{eq2} transfer into \eqref{eq1} and vise versa.

This naturally leads to the question of which results for generalized Weyl algebras can be extended to weak
generalized Weyl algebras and how. The paper \cite{LZ} addresses the problem of classification of simple
finite dimensional modules over generalized Heisenberg algebra and makes some progress in study of
simple weight modules. In this paper we use the ideas and methods from \cite{DGO} to study weight modules
for general weak generalized Weyl algebras. It turns out that there are several subtle differences
between GWAs and wGWAs which should be treated with care. To emphasize these differences and to avoid 
unnecessary higher rank technical complications we restrict to the rank one case. This is still enough 
for our main application, which is generalized Heisenberg algebras. A nice and unusual feature of the 
algebras we consider is existence of simple weight modules with both non-trivial finite dimensional
and infinite dimensional weight spaces, see Example~\ref{ex1257}.

The paper is organized as follows: in Section~\ref{s2} we collect some basic facts on simple
weight modules over generalized Weyl algebras and specify our main restriction on the defining endomorphism,
namely that we always assume that this endomorphism is essential, see Subsection~\ref{s2.3}.
In  Subsection~\ref{s2.4} we define several families of weight modules which we call string modules.
In Section~\ref{s3} we establish the classical decomposition of the category of weight modules.
In Section~\ref{s4} we first describe the support of a simple weight module in Proposition~\ref{prop16}
and then prove our main classification result: Theorem~\ref{thm23}. In Section~\ref{s5} we derive 
some corollaries about finite dimensional modules and, finally, in Section~\ref{s6} we describe applications
to generalized Heisenberg algebras.

\noindent{\bf Acknowledgements.}
R.L. is partially supported by NSF of China  (Grant 11371134) and 
Jiangsu Government Scholarship for Overseas Studies (JS-2013-313).\\
V. M. is partially supported by the Swedish Research Council.\\
K.Z. is partially supported by  NSF of China (Grant 11271109) and NSERC.

\section{Weight modules over weak generalized Weyl algebras}\label{s2}

\subsection{Weight and generalized weight modules}\label{s2.1}

Let $A=A(R,\sigma,t)$ be a wGWA. Let $\mathrm{Max}(R)$ denote the set of all maximal ideals in $R$.
For an $A$-module $M$ and $\mathbf{m}\in\mathrm{Max}(R)$ set
\begin{displaymath}
M_{\mathbf{m}}:=\{v\in M\,:\, \mathbf{m}v=0\},\quad
M^{\mathbf{m}}:=\{v\in M\,:\, \mathbf{m}^kv=0,\, k\gg 0\}.
\end{displaymath}
A module $M$ is called a {\em weight} module provided that 
\begin{displaymath}
M=\bigoplus_{\mathbf{m}\in\mathrm{Max}(R)} M_{\mathbf{m}}.
\end{displaymath}
A module $M$ is called a {\em generalized weight} module provided that 
\begin{displaymath}
M=\bigoplus_{\mathbf{m}\in\mathrm{Max}(R)} M^{\mathbf{m}}.
\end{displaymath}
Clearly, any weight module is a generalized weight module.

For a weight $A$-module $M$ define its {\em support} as the set
\begin{displaymath}
\mathrm{supp}(M):=\{\mathbf{m}\in\mathrm{Max}(R)\,:\, M_{\mathbf{m}}\neq 0\}. 
\end{displaymath}
Similarly one defines the support of a generalized weight module.

For $\mathbf{m}\in \mathrm{Max}(R)$ we denote by $\Bbbk_{\mathbf{m}}$ the quotient field
$R/\mathbf{m}$.

%
%
%
%

\subsection{Action of generators on weight vectors}\label{s2.2}

\begin{lemma}\label{lem1}
Let $M$ be a weight $A$-module, $\mathbf{m}\in \mathrm{supp}(M)$ and $v\in M_{\mathbf{m}}$ be a nonzero element.
Then $Xv=0$ implies $t\in\mathbf{m}$. Furthermore, $Yv=0$ implies $\sigma(t)\in\mathbf{m}$.
\end{lemma}

\begin{proof}
From $Xv=0$ we have $0=YXv=tv$ using \eqref{eq1} which implies $t\in\mathbf{m}$. Similarly,
from $Yv=0$ we have $0=XYv=\sigma(t)v$ (again using \eqref{eq1}) which implies $\sigma(t)\in\mathbf{m}$.
\end{proof}

\begin{lemma}\label{lem2}
Let $M$ be a weight $A$-module and $\mathbf{m}\in \mathrm{supp}(M)$. Then
\begin{displaymath}
X M_{\mathbf{m}} \subset \bigoplus_{\mathbf{n}:\sigma(\mathbf{m})\subset \mathbf{n}} M_{\mathbf{n}}. 
\end{displaymath}
\end{lemma}

\begin{proof}
Let $v$ be a nonzero element in $M_{\mathbf{m}}$. Write $Xv=w_1+w_2+\dots+w_k$ where for each $i$ we have
$\mathbf{n}_i w_i=0$ for some $\mathbf{n}_i \in \mathrm{Max}(R)$. Without loss of generality we may assume 
$w_i\neq 0$ for all $i$ and $\mathbf{n}_i\neq \mathbf{n}_j$ if $i\neq j$. Using \eqref{eq1} we have
\begin{displaymath}
0=X\mathbf{m}v=\sigma(\mathbf{m})Xv= \sum_i \sigma(\mathbf{m})w_i
\end{displaymath}
which means that $\sigma(\mathbf{m})w_i=0$ and hence $\sigma(\mathbf{m})\subset \mathbf{n}_i$ (for each $i$).
The claim follows.
\end{proof}

\begin{lemma}\label{lem3}
Let $M$ be a weight $A$-module and $\mathbf{m}\in \mathrm{supp}(M)$. Assume that one of the following holds:
\begin{enumerate}[$($a$)$]
\item\label{lem3.1} We have $\sigma(t)\not\in\mathbf{m}$.
\item\label{lem3.2} There is $\mathbf{n}\in \mathrm{supp}(M)$ such that $\sigma(\mathbf{n})\subset \mathbf{m}$.
\end{enumerate}
Then
\begin{displaymath}
Y M_{\mathbf{m}} \subset \bigoplus_{\mathbf{n}:\sigma(\mathbf{n})\subset \mathbf{m}} M_{\mathbf{n}}. 
\end{displaymath}
\end{lemma}

\begin{proof}
Let $v$ be a nonzero element in $M_{\mathbf{m}}$. Write $Yv=w_1+w_2+\dots+w_k$ 
where $\mathbf{n}_i w_i=0$ for some $\mathbf{n}_i \in \mathrm{Max}(R)$ (as above, we assume 
$w_i\neq 0$ for all $i$ and $\mathbf{n}_i\neq \mathbf{n}_j$ if $i\neq j$). Using \eqref{eq1} we have
\begin{displaymath}
0=\mathbf{n}_1\mathbf{n}_2\cdots\mathbf{n}_k Yv= Y\sigma(\mathbf{n}_1\mathbf{n}_2\cdots\mathbf{n}_k)v.
\end{displaymath}
If we have $\sigma(t)\not\in\mathbf{m}$ (that is condition \eqref{lem3.1} is satisfied), 
then the action of $Y$ on $M_{\mathbf{m}}$ is injective by Lemma~\ref{lem1} and
hence $\sigma(\mathbf{n}_1\mathbf{n}_2\cdots\mathbf{n}_k)v=0$, that is we have the inclusion
$R\sigma(\mathbf{n}_1\mathbf{n}_2\cdots\mathbf{n}_k)R\subset \mathbf{m}$. Then we have
\begin{displaymath}
R\sigma(\mathbf{n}_1)R\,R\sigma(\mathbf{n}_2)R\,\cdots \,R\sigma(\mathbf{n}_k)R\subset \mathbf{m}. 
\end{displaymath}
As $\mathbf{m}$ is maximal (and hence prime), we get $R\sigma(\mathbf{n}_i)R\subset \mathbf{m}$ 
(which is equivalent to $\sigma(\mathbf{n}_i)\subset \mathbf{m}$) for some $i$. This shows that
condition \eqref{lem3.2} is satisfied.

Let us now prove the claim in the case when condition \eqref{lem3.2} is satisfied.
If there is $\mathbf{n}\in \mathrm{supp}(M)$ such that $\sigma(\mathbf{n})\subset \mathbf{m}$, then from \eqref{eq1}
we have $\mathbf{n}Yv=Y\sigma(\mathbf{n})v=0$, which implies $\mathbf{n}=\mathbf{n}_i$ for all $i$.
The claim follows.
\end{proof}

\begin{corollary}\label{cor4}
Let $\mathbf{m}\in\mathrm{Max}(R)$ be such that $\sigma(t)\not\in \mathbf{m}$ and for any 
$\mathbf{n}\in\mathrm{Max}(R)$ we have $\sigma(\mathbf{n})\not\subset \mathbf{m}$.
Then for any weight module $M$ we have $M_{\mathbf{m}}=0$.
\end{corollary}

\begin{proof}
Assume $v\in M_{\mathbf{m}}$ is nonzero. Then $\sigma(t)v\neq 0$ as $\sigma(t)\not\in \mathbf{m}$
and hence $Yv\neq 0$ either. As for any $\mathbf{n}\in\mathrm{Max}(R)$ we have 
$\sigma(\mathbf{n})\not\subset \mathbf{m}$, we get a contradiction with Lemma~\ref{lem3}.
\end{proof}

\begin{lemma}\label{lem5}
For any $\mathbf{m}\in\mathrm{Max}(R)$ there is at most one $\mathbf{n}\in\mathrm{Max}(R)$ such that
$\sigma(\mathbf{n})\subset \mathbf{m}$.
\end{lemma}

\begin{proof}
Assume $\mathbf{k},\mathbf{n}\in\mathrm{Max}(R)$ are such that $\sigma(\mathbf{k})\subset \mathbf{m}$, 
$\sigma(\mathbf{n})\subset \mathbf{m}$ and $\mathbf{k}\neq\mathbf{n}$. Then $\mathbf{k}+\mathbf{n}=R$
and hence $\sigma(R)=\sigma(\mathbf{k}+\mathbf{n})\subset \mathbf{m}$ which contradicts $\sigma(1)=1$
(recall that $\sigma$ was assumed to be a unital endomorphism of $R$).
\end{proof}

\subsection{Essential endomorphisms}\label{s2.3}

We say that $\sigma$ is {\em essential} provided that for any $\mathbf{m}\in\mathrm{Max}(R)$ the 
canonical inclusion 
\begin{displaymath}
\sigma(R)/(\sigma(R)\cap\mathbf{m})\hookrightarrow R/\mathbf{m} 
\end{displaymath}
is surjective. Obviously, if $\sigma$ is an automorphism, then it is an essential endomorphism.

\begin{example}\label{ex6}
{\rm
Assume that $R$ is a finitely or countably generated algebra over an uncountable algebraically closed field $\Bbbk$
and $\sigma$ is $\Bbbk$-linear.
Then $\Bbbk_{\mathbf{m}}\cong \Bbbk$ for any $\mathbf{m}\in\mathrm{Max}(R)$ (see for example, \cite[Lemma~1]{FGM})
and hence any (unital) endomorphism of $R$ is essential.
} 
\end{example}

\begin{example}\label{ex7}
{\rm
Let $\sigma:\mathbb{C}\to\mathbb{C}$ be a non-surjective ring homomorphism. Then such $\sigma$ is obviously
not essential. Existence of such $\sigma$ follows from the fact that the algebraic closure of the field 
$\mathbb{C}(x)$ of rational functions is isomorphic to 
$\mathbb{C}$ by Steinitz Theorem, see \cite[Page~125]{St}. As $\sigma$ we thus can take the composition
$\mathbb{C}\subsetneq\mathbb{C}(x)\subsetneq\overline{\mathbb{C}(x)}\cong\mathbb{C}$.
} 
\end{example}

\begin{center}
\em From now on we always assume that $\sigma$ is essential.
\end{center}

Our main interest in essential morphisms is the following property which we will use frequently:

\begin{lemma}\label{lem731}
Assume that $\sigma$ is essential. Let $M$ be a weight $A$-module, $\mathbf{m}\in\mathrm{supp}(M)$
and $v\in M_{\mathbf{m}}$. Then $Rv=\sigma(R)v$. 
\end{lemma}

\begin{proof}
We have $Rv=\Bbbk_{\mathbf{m}}v$ since $\mathbf{m}v=0$. Further,  $\Bbbk_{\mathbf{m}}v=\sigma(R)v$
since $\sigma$ is essential. 
\end{proof}

Further, for essential $\sigma$ we can strengthen Lemma~\ref{lem3} as follows:

\begin{lemma}\label{lem3new}
Assume that $\sigma$ is essential. Let $M$ be a weight $A$-module, $\mathbf{m}\in \mathrm{supp}(M)$
and assume that $Y M_{\mathbf{m}}\neq 0$. Then there is a unique  $\mathbf{n}\in \mathrm{supp}(M)$
such that $\sigma(\mathbf{n})\subset \mathbf{m}$, moreover,  $Y M_{\mathbf{m}} \subset M_{\mathbf{n}}$.
\end{lemma}

\begin{proof}
Taking into account Lemmata~\ref{lem3} and  \ref{lem5}, we only need to prove the existence of 
$\mathbf{n}\in \mathrm{supp}(M)$ such that $\sigma(\mathbf{n})\subset \mathbf{m}$.
 
Let $v$ be a nonzero element in $M_{\mathbf{m}}$ such that $Yv\neq 0$. Write $Yv=w_1+w_2+\dots+w_k$ 
where $\mathbf{n}_i w_i=0$ for some $\mathbf{n}_i \in \mathrm{Max}(R)$. We assume 
$w_i\neq 0$ for all $i$ and $\mathbf{n}_i\neq \mathbf{n}_j$ if $i\neq j$. On the one hand,
$\mathbf{n}_k Yv=Y\sigma(\mathbf{n}_k)v$ by \eqref{eq1}. On the other hand,
\begin{displaymath}
\mathbf{n}_k Yv=\mathbf{n}_k(w_1+w_2+\dots+w_k)=\mathbf{n}_kw_1+\mathbf{n}_kw_2+\dots+\mathbf{n}_kw_{k-1}
\end{displaymath}
since $\mathbf{n}_kw_k=0$. Moreover, $\mathbf{n}_kw_i\neq 0$ for all $i<k$ since $\mathbf{n}_k\neq \mathbf{n}_i$.
Therefore we may assume $k=1$ and set $w:=w_1$ and $\mathbf{n}:=\mathbf{n}_1$.

Assume $\sigma(\mathbf{n})\not\subset \mathbf{m}$. We have $0= \mathbf{n}Yv=Y\sigma(\mathbf{n})v$.
Pick $a\in\mathbf{n}$ such that $\sigma(a)\not\in\mathbf{m}$. 
As $\sigma$ is essential, there is $b\in R$ such that 
\begin{equation}\label{eq5}
(\sigma(b)+\mathbf{m})(\sigma(a)+\mathbf{m})=1+\mathbf{m}. 
\end{equation}
On the one hand, we have $baYv=0$ since $aYv=0$. On the other hand,  by \eqref{eq1} we have
\begin{displaymath}
baYv=Y \big(\sigma(b)\sigma(a)v\big)\overset{\eqref{eq5}}{=}Y(1v)=Yv\neq 0
\end{displaymath}
by our assumptions, a contradiction. This completes the proof.
\end{proof}

\subsection{String modules}\label{s2.4}

Denote by ${}_{\infty}Q_{\infty}$ the set of all possible maps $\mathbb{Z}\to\mathrm{max}(R)$,
$i\mapsto \mathbf{m}_i$, such that $\sigma(\mathbf{m}_i)\subset\mathbf{m}_{i+1}$ for all $i\in \mathbb{Z}$.
We will write $\underline{\mathbf{m}}:=\big(\mathbf{m}_i\big)\in {}_{\infty}Q_{\infty}$ if the collection 
$\big(\mathbf{m}_i\big)$ comes from such a map.

For $i\in\mathbb{Z}$ define $X:\Bbbk_{\mathbf{m}_i}\to \Bbbk_{\mathbf{m}_{i+1}}$ as mapping $r+\mathbf{m}_i$ 
to $\sigma(tr)+\mathbf{m}_{i+1}$. Further, define 
$Y:\Bbbk_{\mathbf{m}_i}\to \Bbbk_{\mathbf{m}_{i-1}}$ as mapping $r+\mathbf{m}_i$ to 
$s+\mathbf{m}_{i-1}$ for some $s\in R$ such that $\sigma(s)\in r+\mathbf{m}_i$. Such an $s\in R$ exists 
because $\sigma$ is essential. 
For $\underline{\mathbf{m}}\in {}_{\infty}Q_{\infty}$ set 
\begin{displaymath}
N(\underline{\mathbf{m}}):=\bigoplus_{i\in\mathbb{Z}}\Bbbk_{\mathbf{m}_i}.
\end{displaymath}
Then $N(\underline{\mathbf{m}})$ carries the natural structure of an $R$-module. 
Note that, by construction, the action of $Y$ on $N(\underline{\mathbf{m}})$ defined above
is bijective.

\begin{lemma}\label{lem8}
The above defines on $N(\underline{\mathbf{m}})$ the structure of a weight $A$-module.
\end{lemma}

\begin{proof}
First let us check that the action of $Y$ is well-defined. If $s,s'\in R$ are such that 
$\sigma(s),\sigma(s')\in r+\mathbf{m}_i$, then $\sigma(s-s')\in \mathbf{m}_i$. Since $\sigma$ is unital,
in particular, nonzero and $\Bbbk_{\mathbf{m}_{i-1}}$ is a field, it follows
that $s-s'\in \mathbf{m}_{i-1}$. Therefore the action of $Y$ is well-defined.

It remains to check relations \eqref{eq1}. For $r\in R$ and $i\in\mathbb{Z}$ we have
(using the notation above)
\begin{gather*}
XY(r+\mathbf{m}_i)=X(s+\mathbf{m}_{i-1})=\sigma(ts)+\mathbf{m}_i=
\sigma(t)r+\mathbf{m}_i,\\
YX(r+\mathbf{m}_i)=Y(\sigma(tr)+\mathbf{m}_{i+1})=tr+\mathbf{m}_i.
\end{gather*}
Moreover, for $a\in R$ we also have 
\begin{multline*}
Xa(r+\mathbf{m}_i)=X(ar+\mathbf{m}_{i})=\sigma(tar)+\mathbf{m}_{i+1}=\\=\sigma(a)(\sigma(tr)+\mathbf{m}_{i+1})=
\sigma(a)X(r+\mathbf{m}_i).
\end{multline*}
Finally, for $a\in R$ we have 
\begin{displaymath}
aY(r+\mathbf{m}_i)=a(s+\mathbf{m}_{i-1})=as+\mathbf{m}_{i-1}
\end{displaymath}
and
\begin{displaymath} 
Y\sigma(a)(r+\mathbf{m}_i)=Y(\sigma(a)r+\mathbf{m}_i)=s'+\mathbf{m}_{i-1},
\end{displaymath}
where $s'$ is such that $\sigma(s')+\mathbf{m}_{i}=\sigma(a)r+\mathbf{m}_{i}$. Note that we have
$r+\mathbf{m}_{i}=\sigma(s)+\mathbf{m}_{i}$ and hence
\begin{displaymath}
\sigma(a)r+\mathbf{m}_{i}=\sigma(a)\sigma(s)+\mathbf{m}_{i}=\sigma(as)+\mathbf{m}_{i} 
\end{displaymath}
since $\sigma$ is a homomorphism.
By the previous paragraph, we thus have $s'+\mathbf{m}_{i-1}=as+\mathbf{m}_{i-1}$ and the proof is complete.
\end{proof}

The module $N(\underline{\mathbf{m}})$ will be referred to as a {\em double infinite string module}.
These modules should be compared with modules in \cite[Section~3]{DGO} and \cite[Lemma~7]{LZ}.

\begin{lemma}\label{lemnew1}
For $\underline{\mathbf{m}},\underline{\mathbf{n}}\in {}_{\infty}Q_{\infty}$ we have
$N(\underline{\mathbf{m}})\cong N(\underline{\mathbf{n}})$ if and only if there is $k\in\mathbb{Z}$ such that
$\mathbf{m}_{i}=\mathbf{n}_{i+k}$ for all $i\in\mathbb{Z}$. 
\end{lemma}

\begin{proof}
The ``if'' part is clear. To prove the ``only if'' part we consider an isomorphism 
$\varphi:N(\underline{\mathbf{m}})\to N(\underline{\mathbf{n}})$ and let $v\in \Bbbk_{\mathbf{m}_0}$ be a 
nonzero element. We have $\varphi(v)=\sum_{i\in\mathbb{Z}}w_i$ with $w_i\in \Bbbk_{\mathbf{n}_i}$.
As $\varphi(v)\neq 0$ we may define $l:=\max\{i\in\mathbb{Z}:w_i\neq 0\}$. As $\varphi$ is an isomorphism,
we have $\mathbf{n}_{l}=\mathbf{m}_{0}$. Now the fact that
$\mathbf{n}_{l-i}=\mathbf{m}_{-i}$ for all $i\in\mathbb{N}$ follows from Lemma~\ref{lem3new}.
The fact that $\mathbf{n}_{l+i}=\mathbf{m}_{+i}$ for all $i\in\mathbb{N}$ follows by combining that
$\varphi$ is a isomorphism and that the action of $Y$ is bijective on both $N(\underline{\mathbf{m}})$ and $N(\underline{\mathbf{n}})$.
\end{proof}

Denote by $Q_{\infty}$ the set of all maps $\mathbb{Z}_+:=\{0,1,2,\dots\}\to\mathrm{max}(R)$,
$i\mapsto \mathbf{m}_i$, such that $\sigma(\mathbf{m}_i)\subset\mathbf{m}_{i+1}$ for all $i\in \mathbb{Z}_+$
and, additionally, $\sigma(t)\in \mathbf{m}_0$. We will write $\underline{\mathbf{m}}:=\big(\mathbf{m}_i\big)
\in Q_{\infty}$  if the collection  $\big(\mathbf{m}_i\big)$ comes from such a map.

For $\underline{\mathbf{m}}\in Q_{\infty}$ set 
\begin{displaymath}
N^{\uparrow}(\underline{\mathbf{m}}):=\bigoplus_{i\in\mathbb{Z}_+}\Bbbk_{\mathbf{m}_i}.
\end{displaymath}
Then $N^{\uparrow}(\underline{\mathbf{m}})$ carries the natural structure of an $R$-module. For $i\in\mathbb{Z}_+$
define $X:\Bbbk_{\mathbf{m}_i}\to \Bbbk_{\mathbf{m}_{i+1}}$ as mapping $r+\mathbf{m}_i$ to
$\sigma(tr)+\mathbf{m}_{i+1}$. Finally, set $Y\Bbbk_{\mathbf{m}_0}=0$ and for $i\in\{1,2,3,\dots\}$ define 
$Y:\Bbbk_{\mathbf{m}_i}\to \Bbbk_{\mathbf{m}_{i-1}}$ as mapping $r+\mathbf{m}_i$ to 
$s+\mathbf{m}_{i-1}$ for some $s\in R$ such that $\sigma(s)\in r+\mathbf{m}_i$. Such an $s\in R$ exists 
because $\sigma$ is essential.

\begin{lemma}\label{lem9}
The above defines on $N^{\uparrow}(\underline{\mathbf{m}})$ the structure of a weight $A$-module.
\end{lemma}

\begin{proof}
Mutatis mutandis the proof of Lemma~\ref{lem8}.
\end{proof}

The module $N^{\uparrow}(\underline{\mathbf{m}})$ will be referred to as a {\em right infinite string module}.
Note that some right infinite string modules are subquotients of double infinite string modules.
The weight $\mathbf{m}_0$ will be called the {\em lowest} weight.

\begin{lemma}\label{lemnew2}
For $\underline{\mathbf{m}},\underline{\mathbf{n}}\in Q_{\infty}$ we have
$N^{\uparrow}(\underline{\mathbf{m}})\cong N^{\uparrow}(\underline{\mathbf{n}})$ 
if and only if $\underline{\mathbf{m}}=\underline{\mathbf{n}}$. 
\end{lemma}

\begin{proof}
Mutatis mutandis the proof of Lemma~\ref{lemnew1}.
\end{proof}

Denote by ${}_{\infty}Q$ the set of all maps $\mathbb{Z}_-:=\{0,-1,-2,\dots\}\to\mathrm{max}(R)$,
$i\mapsto \mathbf{m}_i$, such that $\sigma(\mathbf{m}_{i-1})\subset\mathbf{m}_{i}$ for all $i\in \mathbb{Z}_-$
and, additionally, $t\in \mathbf{m}_0$. We will write $\underline{\mathbf{m}}:=\big(\mathbf{m}_i\big)\in {}_{\infty}Q$ 
if the collection  $\big(\mathbf{m}_i\big)$ comes from such a map.

For $\underline{\mathbf{m}}\in {}_{\infty}Q$ set 
\begin{displaymath}
N^{\downarrow}(\underline{\mathbf{m}}):=\bigoplus_{i\in\mathbb{Z}_-}\Bbbk_{\mathbf{m}_i}.
\end{displaymath}
Then $N^{\downarrow}(\underline{\mathbf{m}})$ carries the natural structure of an $R$-module. 
Set $X\Bbbk_{\mathbf{m}_0}=0$ and for $i\in\{-1,-2,-3,\dots\}$
define $X:\Bbbk_{\mathbf{m}_i}\to \Bbbk_{\mathbf{m}_{i+1}}$ as mapping $r+\mathbf{m}_i$ to
$\sigma(tr)+\mathbf{m}_{i+1}$. Finally, for $i\in\mathbb{Z}_-$ define 
$Y:\Bbbk_{\mathbf{m}_i}\to \Bbbk_{\mathbf{m}_{i-1}}$ as mapping $r+\mathbf{m}_i$ to 
$s+\mathbf{m}_{i-1}$ for some $s\in R$ such that $\sigma(s)\in r+\mathbf{m}_i$. Such an $s\in R$ exists 
because $\sigma$ is essential.

\begin{lemma}\label{lem10}
The above defines on $N^{\downarrow}(\underline{\mathbf{m}})$ the structure of a weight $A$-module.
\end{lemma}

\begin{proof}
Mutatis mutandis the proof of Lemma~\ref{lem8}.
\end{proof}

The module $N^{\downarrow}(\underline{\mathbf{m}})$ will be referred to as a {\em left infinite string module}.
Note that some left infinite string modules are subquotients of double infinite string modules.
The weight $\mathbf{m}_0$ will be called the {\em highest} weight.

\begin{lemma}\label{lemnew3}
For $\underline{\mathbf{m}},\underline{\mathbf{n}}\in {}_{\infty}Q$ we have
$N^{\downarrow}(\underline{\mathbf{m}})\cong N^{\downarrow}(\underline{\mathbf{n}})$ 
if and only if $\underline{\mathbf{m}}=\underline{\mathbf{n}}$. 
\end{lemma}

\begin{proof}
Mutatis mutandis the proof of Lemma~\ref{lemnew1}.
\end{proof}

For $n\in\mathbb{Z}_+$ denote by $Q_n$ the set of all maps $\{0,1,\dots,n\}\to\mathrm{max}(R)$,
$i\mapsto \mathbf{m}_i$, such that $\sigma(\mathbf{m}_{i})\subset\mathbf{m}_{i+1}$ for all $i=0,1,\dots,n-1$
and, additionally, $\sigma(t)\in \mathbf{m}_0$ and $t\in\mathbf{m}_n$. We will write $\underline{\mathbf{m}}:=\big(\mathbf{m}_i\big)\in Q_n$  if the collection  $\big(\mathbf{m}_i\big)$ comes from such a map.

For $\underline{\mathbf{m}}\in Q_n$ set 
\begin{displaymath}
N_n(\underline{\mathbf{m}}):=\bigoplus_{i=0}^n\Bbbk_{\mathbf{m}_i}.
\end{displaymath}
Then $N_n(\underline{\mathbf{m}})$ carries the natural structure of an $R$-module. 
Set $X\Bbbk_{\mathbf{m}_n}=0$ and for $i\in\{0,1,\dots,n-1\}$
define $X:\Bbbk_{\mathbf{m}_i}\to \Bbbk_{\mathbf{m}_{i+1}}$ as mapping $r+\mathbf{m}_i$ to
$\sigma(tr)+\mathbf{m}_{i+1}$. Finally, set $Y\Bbbk_{\mathbf{m}_0}=0$ and for $i\in\{1,2,\dots,n\}$ define 
$Y:\Bbbk_{\mathbf{m}_i}\to \Bbbk_{\mathbf{m}_{i-1}}$ as mapping $r+\mathbf{m}_i$ to 
$s+\mathbf{m}_{i-1}$ for some $s\in R$ such that $\sigma(s)\in r+\mathbf{m}_i$. Such an $s\in R$ exists 
because $\sigma$ is essential.

\begin{lemma}\label{lem11}
The above defines on $N_n(\underline{\mathbf{m}})$ the structure of a weight $A$-module.
\end{lemma}

\begin{proof}
Mutatis mutandis the proof of Lemma~\ref{lem8}.
\end{proof}

The module $N_n(\underline{\mathbf{m}})$ will be referred to as a {\em bounded string module}.
Note that some bounded string modules are subquotients of infinite string modules.
The weight $\mathbf{m}_0$ will be called the {\em lowest} weight.
The weight $\mathbf{m}_n$ will be called the {\em highest} weight.

\begin{lemma}\label{lemnew4}
For $\underline{\mathbf{m}},\underline{\mathbf{n}}\in Q_n$ we have
$N_n(\underline{\mathbf{m}})\cong N_n(\underline{\mathbf{n}})$ 
if and only if $\underline{\mathbf{m}}=\underline{\mathbf{n}}$. 
\end{lemma}

\begin{proof}
Mutatis mutandis the proof of Lemma~\ref{lemnew1}.
\end{proof}

\subsection{Simplicity of string modules}\label{s2.5}

As we will see in this subsection, generically the string modules defined above are simple
(this should be compared to \cite[Section~3]{DGO} and \cite[Section~3]{LZ}).

We call $\underline{\mathbf{m}}=\big(\mathbf{m}_i\big)\in{}_{\infty}Q_{\infty}$ {\em periodic} provided that there 
is a nonzero $k\in\mathbb{Z}$ such that $\mathbf{m}_i=\mathbf{m}_{i+k}$ for all $i\in\mathbb{Z}$ and
{\em aperiodic} otherwise.

\begin{lemma}\label{lem12}
Let $\underline{\mathbf{m}}\in {}_{\infty}Q_{\infty}$. Assume that there exist $m,n\in\mathbb{Z}$ such that
$m<n$ and $\mathbf{m}_m=\mathbf{m}_n$. Then $\mathbf{m}_{m+i}=\mathbf{m}_{n+i}$ for all $i\in\mathbb{Z}_-$.
\end{lemma}

\begin{proof}
This follows directly from Lemma~\ref{lem5}. 
\end{proof}

\begin{proposition}\label{prop13}
For $\underline{\mathbf{m}}\in{}_{\infty}Q_{\infty}$ the module $N(\underline{\mathbf{m}})$ is simple
if and only if $\underline{\mathbf{m}}$ is aperiodic and $t\not\in \mathbf{m}_i$  
for all $i\in\mathbb{Z}$.
\end{proposition}

\begin{proof}
We start with the ``if'' statement. If $\underline{\mathbf{m}}$ is aperiodic, then from Lemma~\ref{lem12} it 
follows that $\mathbf{m}_i\neq \mathbf{m}_j$ for all $i,j\gg0$. Let $v\in N(\underline{\mathbf{m}})$ be a nonzero 
weight vector. If $t\not\in \mathbf{m}_i$ for all $i\in\mathbb{Z}$, then the action of $X$ on $N(\underline{\mathbf{m}})$ is injective by Lemma~\ref{lem1}. For $k\gg 0$ weight components of $X^kv$ thus 
are of the form $r+\mathbf{m}_i$ for some $i\gg0$. In particular, they all are different. Hence the
submodule of $N(\underline{\mathbf{m}})$ generated by $v$ contains 
$N(\underline{\mathbf{m}})_{\mathbf{m}_i}$ for some $i\gg 0$. 
Note further that $t\not\in \mathbf{m}_i$ for all $i\in\mathbb{Z}$ implies
$\sigma(t)\not\in \mathbf{m}_i$ for all $i\in\mathbb{Z}$. Hence the action of $Y$ on $N(\underline{\mathbf{m}})$ is injective as well by Lemma~\ref{lem1}. The elements $X^k(r+\mathbf{m}_i)$ and $Y^k(r+\mathbf{m}_i)$, 
$k\in\mathbb{Z}_+$, generate $N(\underline{\mathbf{m}})$. Therefore the submodule of $N(\underline{\mathbf{m}})$
generated by $v$ is the whole  $N(\underline{\mathbf{m}})$ which proves that $N(\underline{\mathbf{m}})$ is simple.

Now we prove the ``only if'' statement. If $t\in \mathbf{m}_i$ for some $i$, then 
\begin{displaymath}
\bigoplus_{j\leq i}\Bbbk_{\mathbf{m}_i} 
\end{displaymath}
is a proper submodule of $N(\underline{\mathbf{m}})$ and hence $N(\underline{\mathbf{m}})$ is not simple.

If $\underline{\mathbf{m}}$ is periodic, say with some period $0\neq k\in\mathbb{Z}$, then it is easy to
check that $(1+\mathbf{m}_0)+(1+\mathbf{m}_k)$ generates a proper submodule in $N(\underline{\mathbf{m}})$
and hence $N(\underline{\mathbf{m}})$ is not simple.
\end{proof}

\begin{proposition}\label{prop14}
\begin{enumerate}[$($i$)$]
\item\label{prop14.1} For $\underline{\mathbf{m}}\in Q_{\infty}$ the module $N^{\uparrow}(\underline{\mathbf{m}})$ 
is simple if and only if $t\not\in\mathbf{m}_i$ for all $i\in\mathbb{Z}_+$.
\item\label{prop14.2} For $\underline{\mathbf{m}}\in {}_{\infty}Q$ the module $N^{\downarrow}(\underline{\mathbf{m}})$ 
is simple if and only if $t\not\in\mathbf{m}_i$ for all $i<0$.
\item\label{prop14.3} For $\underline{\mathbf{m}}\in Q_n$ the module $N_n(\underline{\mathbf{m}})$ 
is simple if and only if $t\not\in\mathbf{m}_i$ for all $i=0,1,2,\dots,{n-1}$.
\end{enumerate}
\end{proposition}

\begin{proof}
Mutatis mutandis the proof of Proposition~\ref{prop13}.
\end{proof}

\section{Category of weight modules}\label{s3}

\subsection{``Orbits'' of $\sigma$}\label{s3.1}

For $\mathbf{n},\mathbf{m}\in\mathrm{Max}(R)$ write $\mathbf{n}\leftarrow\mathbf{m}$ provided that
$\sigma(\mathbf{n})\subset\mathbf{m}$. Let $\sim$ denote the minimal equivalence relation on $\mathrm{Max}(R)$ 
which contains the relation $\leftarrow$. 

Because of Lemma~\ref{lem5} the relation $\leftarrow$ can be interpreted as a partially defined endomorphism 
(operating from the right to the left) of the set $\mathrm{Max}(R)$. With this interpretation,
equivalence classes of $\sim$ can be though of as some kind of generalized ``orbits'' of the map $\leftarrow$.
If $\sigma$ is an automorphism, then equivalence classes of $\sim$ are genuine orbits of the cyclic group
generated by $\sigma$ acting on $\mathrm{Max}(R)$.

For $\mathbf{n},\mathbf{m}\in\mathrm{Max}(R)$ write $\mathbf{n}\overset{0}{\leftarrow}\mathbf{m}$
provided that $\mathbf{n}=\mathbf{m}$. For $k\in\mathbb{N}$ write
$\mathbf{n}\overset{k}{\leftarrow}\mathbf{m}$ provided that there exist 
$\mathbf{m}_i\in\mathrm{Max}(R)$, $i=1,2,\dots,k-1$, such that
\begin{displaymath}
\mathbf{n}\leftarrow\mathbf{m}_1\leftarrow\mathbf{m}_2\leftarrow\dots
\leftarrow\mathbf{m}_{k-1}\leftarrow\mathbf{m}.
\end{displaymath}
In particular, $\mathbf{n}\overset{1}{\leftarrow}\mathbf{m}$ if and only if $\mathbf{n}\leftarrow\mathbf{m}$.
Further, we have $\mathbf{n}\sim\mathbf{m}$ if and only if
there is $\mathbf{k}\in\mathrm{Max}(R)$ and $m,n\in\mathbb{Z}_+$
such that $\mathbf{k}\overset{m}{\leftarrow}\mathbf{m}$ and 
$\mathbf{k}\overset{n}{\leftarrow}\mathbf{n}$.
The equivalence classes $S\in \mathrm{Max}(R)/\sim$ fall into two different cases.

{\bf Rooted classes}, that is classes $S$ for which there exists $\mathbf{n}\in S$ such that 
$\sigma(\mathbf{k})\not\subset\mathbf{n}$ for all $\mathbf{k}\in\mathrm{Max}(R)$. In this case 
$S$ consists of all $\mathbf{m}\in\mathrm{Max}(R)$ such that $\mathbf{n}\overset{k}{\leftarrow}\mathbf{m}$
for some $k\in\mathbb{Z}_+$.

{\bf Unrooted classes}, that is classes $S$ for which $\mathbf{n}$ as in the previous paragraph does not exist.
In this case for any $\mathbf{n}\in S$ the class $S$ consists of all $\mathbf{m}\in\mathrm{Max}(R)$ 
for which there exist $\mathbf{k}\in S$ and $k,l\in\mathbb{Z}_+$ such that
$\mathbf{k}\overset{k}{\leftarrow}\mathbf{n}$ and $\mathbf{k}\overset{l}{\leftarrow}\mathbf{m}$.

From this point of view each $S\in \mathrm{Max}(R)/\sim$ is a lower semilattice.

\subsection{Standard decomposition of the category of weight modules}\label{s3.2}

Let $\mathfrak{W}$ denote the full subcategory of $A\text{-}\mathrm{mod}$ consisting of all weight modules.
For $S\in \mathrm{Max}(R)/\sim$ denote by $\mathfrak{W}_S$ the full subcategory of $\mathfrak{W}$ consisting
of all modules $M$ such that $\mathrm{supp}(M)\subset S$. Then we have the following standard fact
(cf. \cite[Proposition~1.5]{DGO}):

\begin{lemma}\label{lem15}
Assume $\sigma$ is essential. Then we have the decomposition
\begin{displaymath}
\mathfrak{W}\cong\bigoplus_{S\in \mathrm{Max}(R)/\sim} \mathfrak{W}_S.
\end{displaymath}
\end{lemma}

\begin{proof}
This follows directly from Lemmata~\ref{lem2} and  \ref{lem3new}.
\end{proof}

As we will see, one major difference between GWAs and wGWAs is that for the former algebras the categories
$\mathfrak{W}_S$ are usually indecomposable, while for the latter algebras the categories
$\mathfrak{W}_S$ are usually decomposable.

\section{Simple weight modules}\label{s4}

\subsection{Support of a simple weight module}\label{s4.1}

\begin{proposition}\label{prop16}
Let $M$ be a simple weight module. Then there is
\begin{displaymath}
\underline{\mathbf{m}}\in  {}_{\infty}Q_{\infty}\bigcup Q_{\infty}\bigcup {}_{\infty}Q\bigcup
\bigcup_{n\in\mathbb{Z}^+}Q_n
\end{displaymath}
such that $\mathrm{supp}(M)\subset \underline{\mathbf{m}}$.
\end{proposition}

\begin{proof}
We start with the easy observation that
\begin{equation}\label{eq7}
A=\sum_{i,j\in\mathbb{Z}_+}Y^iRX^j 
\end{equation}
which follows directly from \eqref{eq1} (note that the sum in \eqref{eq7} is not claimed to be direct).
Now observe the following:

\begin{lemma}\label{lem753}
Let $M$ be a simple weight module. Then for any different $\mathbf{m},\mathbf{n}\in \mathrm{supp}(M)$ there exists
$k\in\mathbb{Z}_+$ and $\mathbf{k}_1,\mathbf{k}_2,\dots,\mathbf{k}_k\in \mathrm{supp}(M)$ such that 
either we have $\mathbf{m}\leftarrow \mathbf{k}_1\leftarrow\dots \leftarrow\mathbf{k}_k\leftarrow\mathbf{n}$
or we have $\mathbf{n}\leftarrow \mathbf{k}_1\leftarrow\dots \leftarrow\mathbf{k}_k\leftarrow\mathbf{m}$.
\end{lemma}

\begin{proof}
This follows immediately combining \eqref{eq7}, Lemma~\ref{lem2}, Lemma~\ref{lem3new} and the property 
$AM_{\mathbf{m}}=M$.
\end{proof}

From \eqref{eq7} it follows that $\mathrm{supp}(M)$ is at most countable, say
$\mathrm{supp}(M)=\{\mathbf{n}_1,\mathbf{n}_2,\dots\}$. Set $\mathbf{m}_1:=\mathbf{n}_1$ and define
$K_1$ as the maximal sequence $(\dots, \mathbf{m}_{-1},\mathbf{m}_0,\mathbf{m}_1)$,
finite or infinite, such that all $\mathbf{m}_{i}\in \mathrm{supp}(M)$ and
$\mathbf{m}_{i}\leftarrow \mathbf{m}_{i+1}$ for all $i\leq 0$ appearing in the sequence. 
Define now $K_s$ for $s>1$ inductively
as follows:  If all elements of $\mathrm{supp}(M)$ appear already in $K_{s-1}$, set $K_s:=K_{s-1}$.
Otherwise, choose $i$ minimal such that $\mathbf{n}_{i}\in \mathrm{supp}(M)$ does not appear in
$K_{s-1}$. Let $\mathbf{m}$ be the rightmost element of $K_{s-1}$. Then from Lemma~\ref{lem753} it follows that 
\begin{displaymath}
\mathbf{m}\leftarrow \mathbf{k}_1\leftarrow\dots \leftarrow\mathbf{k}_k\leftarrow\mathbf{n}_i 
\end{displaymath}
for some $k\in\mathbb{Z}_+$ and $\mathbf{k}_1,\mathbf{k}_2,\dots,\mathbf{k}_k\in \mathrm{supp}(M)$
and we can set
\begin{displaymath}
K_s:=(K_{s-1},\mathbf{k}_1,\dots,\mathbf{k}_k,\mathbf{n}_i).
\end{displaymath}
Then $\mathrm{supp}(M)$ is, by construction, the limit of $K_i$ when $i\to\infty$, and it obviously has the
required form.
\end{proof}

\subsection{Band modules}\label{s4.2}

Let $\underline{\mathbf{m}}=\{\mathbf{m}_i\}$ be doubly infinite and periodic with the minimal period
$k\in\{1,2,3,\dots\}$. Set $\mathbf{m}:=\mathbf{m}_0$ and $\Bbbk:=\Bbbk_{\mathbf{m}}$.
Note that $\sigma$, being essential, induces for each $i>0$ an isomorphism $\Bbbk\cong \Bbbk_{\mathbf{m}_i}$
by sending $r+\mathbf{m}$ to $\sigma^i(r)+\mathbf{m}_i$. Denote the latter isomorphism by $\tau_i$
and also denote by $\overline{\sigma}$ the automorphism of $\Bbbk$ induced by $\tau_k$. 
Set $\tau_0:=\mathrm{id}_{\Bbbk}$.

Consider the skew Laurent polynomial ring $\mathbf{P}:=\Bbbk[\alpha,\alpha^{-1},\overline{\sigma}]$
given by $\alpha (r+\mathbf{m})=\overline{\sigma}(r+\mathbf{m})\alpha$ for $r\in R$. This ring is a
principal ideal domain (both left and right). For an irreducible element  $f\in \mathbf{P}$ denote by 
$L_f:=\mathbf{P}/\mathbf{P}f$ the corresponding simple module. Note that $L_f\cong L_{\tilde{f}}$ if and only if
$f$ and $\tilde{f}$ are similar in $\mathbf{P}$. Note also that $L_f$ is finite dimensional over $\Bbbk$.

Define the weight $A$-module $M^f$ with 
\begin{displaymath}
\mathrm{supp}(M^f):=\{\mathbf{m}_{-(k-1)},\dots\mathbf{m}_{-2},\mathbf{m}_{-1},\mathbf{m}_0\} 
\end{displaymath}
as follows: for $i\in\{0,1,\dots,k-1\}$ define the $R$-module $M^f_{\mathbf{m}_{-i}}$ as
${}^{\tau_i}\hspace{-1mm}L_f$, that is $M^f_{\mathbf{m}_{-i}}=L_f$ as the vector space but the action of $R$ is 
twisted by $\tau_i$, that is $(r+\mathbf{m}_{-i})\cdot v=\tau_i(r+\mathbf{m})v$ for $r\in R$ and $v\in L_f$.
For $i\in\{0,1,\dots,k-2\}$ define the action of $Y$ from $L_f=M^f_{\mathbf{m}_{-i}}$
to $L_f=M^f_{\mathbf{m}_{-i-1}}$ as the identity. 
Define the action of $Y$ from  $M^f_{\mathbf{m}_{-(k-1)}}$ to $M^f_{\mathbf{m}_{0}}$ as the action of 
$\alpha^{-1}$. In particular, the action of $Y$ on $M^f_{\mathbf{m}_{-i}}$ is bijective. Define the action 
of $X$ on $Yv$ as $\sigma(t)v$.

\begin{lemma}\label{lem21}
The above defines on $M^f$ the structure of a simple weight $A$-module,
moreover, $M^f\cong M^{\tilde{f}}$ if and only if $f$ and $\tilde{f}$ are similar.
\end{lemma}

\begin{proof}
To prove that  $M^f$ is an $A$-module, we have to check the defining relations. They all are immediate from the
definitions. To prove simplicity let $v\in M^f$ be a nonzero weight element. As the action of $Y$ is bijective,
we may assume $v\in M^f_{\mathbf{m}}$. The linear operator $Y^k$ on $M^f_{\mathbf{m}}$ is bijective by definition
and hence defines on $M^f_{\mathbf{m}}$ the structure of a $\mathbf{P}$-module which is isomorphic to 
$L_f$ by construction. Therefore the $A$-submodule  generated by $v$ contains $M^f_{\mathbf{m}}$ and thus the 
whole of $M^f$ as the action of $Y$ is bijective.  The claim about isomorphism is clear.
\end{proof}

Define the weight $A$-module $N^f$ with 
\begin{displaymath}
\mathrm{supp}(N^f):=\{\mathbf{m}_0,\mathbf{m}_1,\mathbf{m}_2,\dots,\mathbf{m}_{k-1}\} 
\end{displaymath}
as follows: for $i\in\{0,1,\dots,k-1\}$ define the $R$-module $N^f_{\mathbf{m}_i}$ as
${}^{\tau_i^{-1}}\hspace{-1mm}L_f$, that is $M^f_{\mathbf{m}_i}=L_f$ 
as the vector space but the action of $R$ is 
twisted by $\tau_i^{-1}$, that is $\tau_i(r+\mathbf{m})\cdot v=(r+\mathbf{m})v$ for $r\in R$ and $v\in L_f$.
For $i=0,1,2,\dots,k-2$ define the action of $X$ from $M^f_{\mathbf{m}_i}=L_f$ to 
$M^f_{\mathbf{m}_{i+1}}=L_f$ as the identity. Define the action of $X$ from 
$M^f_{\mathbf{m}_{k-1}}$ to $M^f_{\mathbf{m}_{0}}$ as the action of 
$\alpha^{-1}$. Then the action of $X$ is bijective. Define the action of $Y$ on $Xv$ as $tv$.

\begin{lemma}\label{lem22}
The above defines on $N^f$ the structure of a simple weight $A$-module,
moreover, $N^f\cong N^{\tilde{f}}$ if and only if $f$ and $\tilde{f}$ are similar.
\end{lemma}

\begin{proof}
Mutatis mutandis the proof of Lemma~\ref{lem21}.
\end{proof}

\begin{lemma}\label{lem25}
We have $M^f\cong N^{\tilde{f}}$ if and only if $f$ is similar to $\tilde{f}$ and 
for every $i=0,1,2,\dots,k-1$ we have $t\not\in\mathbf{m}_i$.
\end{lemma}

\begin{proof}
If  $M^f\cong N^{\tilde{f}}$ the both $X$ and $Y$ act bijectively on $M^f$
and hence for every $i=0,1,2,\dots,k-1$ we have $t\not\in\mathbf{m}_i$.
Further, we have $M^f_{\mathbf{m}}\cong N^{\tilde{f}}_{\mathbf{m}}$ as modules over the
subalgebra generated by $R$ and $X^k$. It is easy to check that the latter action defines on
$M^f_{\mathbf{m}}$ and $N^{\tilde{f}}_{\mathbf{m}}$ the structures of $\mathbf{P}$-modules (given by $f$
and $\tilde{f}$, respectively) which thus must be isomorphic. This yields that $f$ and $\tilde{f}$
are similar.

Assume now that $f$ is similar to $\tilde{f}$ and for every $i=0,1,2,\dots,k-1$ we have $t\not\in\mathbf{m}_i$. 
Then the actions of both $X$ and $Y$ on both $M^f$ and $N^{\tilde{f}}$ are bijective. As the action of 
$X^kY^k$ on both $M^f_{\mathbf{m}}$ and $N^{\tilde{f}}_{\mathbf{m}}$ is given by some element of $r$
and is non-zero, $X^k$ is an inverse of $Y^k$ up to a scalar. The claim $M^f\cong N^{\tilde{f}}$ now follows
by comparing the definitions of $M^f$ and $N^{\tilde{f}}$.
\end{proof}

\subsection{Classification of simple weight modules}\label{s4.3}

Our main result is the following theorem

\begin{theorem}\label{thm23}
Each simple weight $A$-module is either a string module or a band module.
\end{theorem}

To prove this theorem we will need several lemmata.

\begin{lemma}\label{lem24}
The action of $X$ on a simple weight $A$-module is either injective or locally nilpotent. 
\end{lemma}

\begin{proof}
Assume that $V$ is a simple weight $A$-module and
$v\in V$ is nonzero and such that $Xv=0$. Write $v=\sum_{\mathbf{m}}v_{\mathbf{m}}$,
where $v_{\mathbf{m}}\in V_{\mathbf{m}}$. From Lemmata~\ref{lem2} and \ref{lem5} it follows that 
for $\mathbf{m}\neq\mathbf{n}$ the weights of the vectors $Xv_{\mathbf{m}}$ and $Xv_{\mathbf{n}}$ do 
not intersect. Hence $Xv=0$ implies $Xv_{\mathbf{m}}=0$ for all ${\mathbf{m}}$. Therefore we may assume
that $v$ is a weight vector, say of weight $\mathbf{m}$.

From \eqref{eq7} it then follows that $V=Av=\sum_{i\in\mathbb{Z}_+}Y^i\Bbbk_{\mathbf{m}}v$. From
\eqref{eq1} and $Xv=0$ we obtain that $X^{i+1}Y^i\Bbbk_{\mathbf{m}}v=0$. The claim follows.
\end{proof}

\begin{lemma}\label{lem25-01}
The action of $Y$ on a simple weight $A$-module is either injective or locally nilpotent. 
\end{lemma}

\begin{proof}
Let $V$ be a simple weight $A$-module. We have a rough description of $\mathrm{supp}(M)$ given in 
Proposition~\ref{prop16}. According to this description, we have one of the following two cases.

{\bf Case 1.} Assume that for any $\mathbf{m}\in \mathrm{supp}(M)$ there exists at most one
$\mathbf{n}\in \mathrm{supp}(M)$ such that $\sigma(\mathbf{m})\subset\mathbf{n}$. 
Assume that $v\in V$ is nonzero and such that $Yv=0$. Similarly to the proof of Lemma~\ref{lem24},
we may assume that $v$ is weight. We claim that $V=\sum_{i\in\mathbb{Z}_+}X^i\Bbbk_{\mathbf{m}}v$.
Denote the right hand side by $N$, it is clearly nonzero and stable under the action of $X$.
From our assumption for Case~1 we also have that all $X^iv$ are weight vectors. As $\sigma$ is
essential, using \eqref{eq1} it follows that $RX^iv=X^i\Bbbk_{\mathbf{m}}v$ and hence
$N$ is stable under the action of $R$. The latter and $Yv=0$ implies
that $N$ is a non-trivial $A$-submodule of $V$ and hence coincides with $V$ as $V$ is simple.
It is easy to check that $Y$ acts locally nilpotently on $N$.

{\bf Case 2.} There is a unique $\mathbf{m}\in \mathrm{supp}(M)$ for which there is more than one
$\mathbf{n}\in \mathrm{supp}(M)$ such that $\sigma(\mathbf{m})\subset\mathbf{n}$, more precisely,
there are exactly two different $\mathbf{n}_1,\mathbf{n}_2\in \mathrm{supp}(M)$ such that 
$\sigma(\mathbf{m})\subset\mathbf{n}_1$ and $\sigma(\mathbf{m})\subset\mathbf{n}_2$. 
Assume that $v\in V$ is nonzero and such that $Yv=0$. We want to show that we may assume that 
$v$ is weight. The only case when for this statement the argument used before does not work is
when $v=v_{1}+v_2$ with $v_i\in V_{\mathbf{n}_i}$, $i=1,2$. We nevertheless claim that 
$Yv_i=0$ for $i=1,2$. Indeed, if not, then we have $Yv_1=-Yv_2$. As $\sigma$ is essential, 
we can rewrite this as $Yv_1=Yrv_2$ for some $r\in R$. Applying $X$ we get $XYv_1=XYrv_2$.
We have $XYv_1\in V_{\mathbf{n}_1}$ and $XYrv_2\in V_{\mathbf{n}_2}$ which implies 
$XYv_1=XYrv_2=0$ since $\mathbf{n}_1\neq \mathbf{n}_2$. If $Yv_1\neq 0$, then from Lemma~\ref{lem24}
we have that $X$ acts locally nilpotently on $V$ and $V=\sum_{i\in\mathbb{Z}_+}Y^i\Bbbk_{\mathbf{m}}Yv_1$.
Combining Proposition~\ref{prop16} with Lemma~\ref{lem3new} it thus follows that $V$ falls into
Case~1, a contradiction. Therefore we may assume that $v$ is a weight vector.

From Lemma~\ref{lem1} it thus follows that there exists $\mathbf{k}\in\mathrm{supp}(M)$ such that 
$\sigma(t)\in \mathbf{k}$. If $Y M_{\mathbf{k}}=0$, the action of $Y$ on $M$ is clearly locally
nilpotent. If $Y M_{\mathbf{k}}\neq 0$, then $Y M_{\mathbf{k}}\subset M_{\mathbf{l}}$
for some $\mathbf{l}\in\mathrm{supp}(M)$ and there is a nonzero $w\in M_{\mathbf{l}}$ such that
$Xv=0$. From the proof of Lemma~\ref{lem24} it follows that in this case $M=\sum_{i\geq 0}Y^iRM_{\mathbf{l}}$.
As one of the weights $\mathbf{n}_1$ or $\mathbf{n}_2$ cannot belong to the support of 
$\sum_{i\geq 0}Y^iRM_{\mathbf{l}}$, this case does not occur. The claim of the lemma follows.
\end{proof}

\begin{proof}[Proof of Theorem~\ref{thm23}.]
Let $V$ be a simple weight $A$-module. We prove the result using a case-by-case analysis.

{\bf Case~1.} Assume that the action of both $X$ and $Y$ on $V$ is locally nilpotent. Let further
$v\in V$ be a weight element such that $Xv=0$. Then from the proof of Lemma~\ref{lem24} we have that 
$V=\sum_{i=0}^k Y^iRv$ for some minimal $k$ and it is straightforward to check that $V$ is a bounded
string module.

{\bf Case~2.} Assume that the action of $X$ on $V$ is locally nilpotent while the action of $Y$ on $V$
is injective. Let $v\in V$ be a weight element, say of weight $\mathbf{m}$, such that $Xv=0$.

{\bf Subcase~2a.} Assume further that for each $i>0$ the weight of the weight element $Y^iv$ is different
from $\mathbf{m}$. Let $I$ be the left ideal of $A$ generated by $\mathbf{m}$ and $X$. Then $A/I$
surjects onto $V$ via the unique map sending $1+I$ to $v$. At the same time, using the action of $A$
on $A/I$ in the basis $\{Y^iv:i\in\mathbb{Z}_+\}$ it is clear that $A/I$ is a right bounded string
module. Moreover, $A/I$ is simple for otherwise it would have a string submodule and the quotient
(which would also surject onto $V$) would be a bounded string module.

{\bf Subcase~2b.} Assume that there is a minimal $k>0$ such that the weight of the weight element 
$Y^kv$ equals $\mathbf{m}$. Let $I$ be the left ideal of $A$ generated by $\mathbf{m}$ and $X$. Then $A/I$
surjects onto $V$ via the unique map sending $1+I$ to $v$. Using the action of $A$
on $A/I$ in the basis $\{Y^iv:i\in\mathbb{Z}_+\}$ it is clear that $A/I$ is a right bounded string
module. However, this string module is not simple as $t\in \mathbf{m}$. Therefore $V$ is a quotient of 
the weight module $A/I$ modulo some proper weight submodule. This implies that for every 
$\mathbf{n}\in \mathrm{supp}(V)$ the space $V_\mathbf{n}$ is finite dimensional over $\Bbbk_{\mathbf{n}}$.
As the action of $Y$ is injective, it follows that all weight spaces have the same dimension and thus
the action of $Y$ is bijective. Now it clear that $V$ is a band module.

{\bf Case~3.} Assume that the action of $Y$ on $V$ is locally nilpotent while the action of $X$ on $V$
is injective. Let $v\in V$ be a weight element, say of weight $\mathbf{m}$, such that $Yv=0$.

{\bf Subcase~3a.} Assume that $\mathrm{supp}(V)$ is infinite. Then from Proposition~\ref{prop16} it follows
that for each $i\in\mathbb{Z}_+$ the element $X^{i}v$ is a weight element, say of weight $\mathbf{m}_i$, and,
moreover, $\mathbf{m}_i\neq \mathbf{m}_j$ if $i\neq j$. The linear span of $RX^{i}v$, $i\in\mathbb{Z}_+$, is 
invariant with respect to the action of $Y$ and hence coincides with $V$. Therefore $V$ is a left bounded 
string module.

{\bf Subcase~3b.} Assume that $\mathrm{supp}(V)$ is finite. Then, using the same arguments to the ones used
for Subcase~2b, one shows that $V$ is a band module.

{\bf Case~4.} Assume that the action of both $Y$ and $X$ on $V$ are injective. 

{\bf Subcase~4a.} Assume that $\mathrm{supp}(V)$ is infinite. Using Proposition~\ref{prop16} we may choose
$\mathbf{m}\in \mathrm{supp}(V)$ such that for any $i\in\mathbb{N}$ the weight of the weight element $Y^iv$,
where $v$ is a nonzero element in $V_{\mathbf{m}}$ is different from $\mathbf{m}$. Then, again from
Proposition~\ref{prop16}, it follows that $X^iv$ is a weight vector for each $i\in\mathbb{N}$. Using the
fact that $\sigma$ is essential one now checks that the linear span of $v$ and $X^iv,Y^iv$, $i\in\mathbb{N}$,
is a submodule and hence coincides with $V$. It follows that $V$ is a doubly infinite string module.

{\bf Subcase~4b.} Assume that $\mathrm{supp}(V)$ is finite. Then, using the same arguments to the ones used
for Subcase~2b, one shows that $V$ is a band module. The claim of the theorem follows.
\end{proof}

\subsection{Weight modules that are of finite length over $R$}\label{s4.4}

Let $V$ be a simple weight $A$-module. A natural analogue of finite dimensionality for $V$
is the requirement that $V$ is of finite length when viewed as an $R$-module. Directly from Theorem~\ref{thm23}
we get:

\begin{corollary}\label{cor24}
The only simple weight $A$-modules which are of finite length over $R$
are bounded string modules and band modules. 
\end{corollary}

\section{Simple finite dimensional modules}\label{s5}

\subsection{Setup and definition}\label{s5.1}

Let $A$ be a wGWA. In this section we assume that $R$ is an algebra over some field $\Bbbk$ and
that $\sigma$ is an algebra endomorphism, in particular, that it is $\Bbbk$-linear.
Then every $A$-module $M$ carries the natural structure of a vector space over $\Bbbk$ and 
we say that $M$ is  {\em finite dimensional} if it  is finite dimensional over $\Bbbk$. 

From the definition of a finite dimensional $A$-module it is clear that each 
finite dimensional $A$-module is a generalized weight module and that the category of
finite dimensional $A$-modules is Krull-Schmidt. In Theorem~\ref{thm41} below we will show that
simple finite dimensional $A$-modules are, in fact, weight modules.

\subsection{Simple finite dimensional modules}\label{s5.2}

Our main result in this section is the following statement which reduces classification of simple
finite dimensional $A$-module to Theorem~\ref{thm23}.

\begin{theorem}\label{thm41}
Assume that $R$ is an algebra over some field $\Bbbk$ and that $\sigma$ is essential and
$\Bbbk$-linear. Then every simple  finite dimensional $A$-module is a weight module.
\end{theorem}

\subsection{Preliminary observations}\label{s5.3}

As usual, we assume that $\sigma$ is essential.
Let $M$ be a finite dimensional $A$-module. Then we have
\begin{displaymath}
M=\bigoplus_{\mathbf{m}\in\mathrm{supp}(M)}M^{\mathbf{m}}, 
\end{displaymath}
$\mathrm{supp}(M)$ is finite and, moreover, each $M^{\mathbf{m}}$ is finite dimensional over $\Bbbk$, in particular,
there is $k\in\mathbb{N}$ such that $M^{\mathbf{m}}$ is annihilated by $\mathbf{m}^k$. Further, 
using the fact that maximal ideals are prime, we obtain
\begin{equation}\label{eq42}
X M^{\mathbf{m}} \subset \bigoplus_{\mathbf{n}:\sigma(\mathbf{m})\subset \mathbf{n}} M^{\mathbf{n}}
\quad\text{ and }\quad  
Y M^{\mathbf{m}} \subset \bigoplus_{\mathbf{n}:\sigma(\mathbf{n})\subset \mathbf{m}} M^{\mathbf{n}}
\end{equation}
by exactly the same arguments as used in the proofs of  Lemma~\ref{lem2} and Lemma~\ref{lem5}, respectively.
From Lemma~\ref{lem5} we know that the direct sum on the right in \eqref{eq42} has at most one summand.
Similarly to Lemma~\ref{lem1} we have that $Xv=0$ for some $v\in M^{\mathbf{m}}$, $v\neq 0$, implies
$t\in\mathbf{m}$ and, analogously,  $Yv=0$ implies $\sigma(t)\in\mathbf{m}$.

Note that, if $\sigma(\mathbf{n})\subset \mathbf{m}$, then $Y$ maps $M^{\mathbf{m}}$ to $M^{\mathbf{n}}$
by the above and, moreover, it maps $M_{\mathbf{m}}$ to $M_{\mathbf{n}}$ by the obvious computation
using \eqref{eq1}. In other words, the action of $Y$ preserves the weight part 
\begin{displaymath}
\mathrm{wt}(M):= \bigoplus_{\mathbf{m}\in\mathrm{supp}(M)}M_{\mathbf{m}}
\end{displaymath}
of $M$.

Let $\Gamma$ be the finite oriented graph whose vertices are elements of the finite set $\mathrm{supp}(M)$
and arrows $\mathbf{n}\leftarrow \mathbf{m}$ are defined as in Subsection~\ref{s3.1}, i.e. represent
the inclusion $\sigma(\mathbf{n})\subset \mathbf{m}$. By  Lemma~\ref{lem8}, each vertex of our graph is
the source of at most one arrow, in particular, the number of arrows does not exceed the number of vertices.

\subsection{Proof of Theorem~\ref{thm41}}\label{s5.4}

Consider first the case when $\Gamma$ contains some vertex $\mathbf{n}$ which is not the target of any 
arrow in $\Gamma$. Let $v\in M_{\mathbf{m}}$ be non-zero. Then $Xv=0$ by \eqref{eq42} and hence \eqref{eq7}
implies 
\begin{displaymath}
M=\sum_{i\in\mathbb{Z}_+}Y^iRv.
\end{displaymath}
Now $Rv\subset M_{\mathbf{m}}$ and since the action of $Y$ preserves $\mathrm{wt}(M)$, we obtain
that $M=\mathrm{wt}(M)$ is a weight module.

Assume now that each vertex in $\Gamma$ is the target of some arrow. Then the number of arrows and
vertices in $\Gamma$ coincide and hence each vertex is the source of exactly one arrow and the target
of exactly one arrow. From \eqref{eq42} and simplicity of $M$ it follows that  $\Gamma$ is an oriented cycle,
say
\begin{displaymath}
\xymatrix{
\mathbf{n}_1\ar@/^15pt/[rrrr]& \mathbf{n}_2\ar[l]& \mathbf{n}_3\ar[l]& \dots\ar[l]
& \mathbf{n}_k.\ar[l]
}
\end{displaymath}
For $i=1,2,\dots,k$ set $l_i:=\dim_{\Bbbk}M_{\mathbf{n}_i}$.

If some $M^{\mathbf{n}_i}$ contains a non-zero $v$ such that $Xv=0$, then we have $XRv=\sigma(R)Xv=0$ by \eqref{eq1}.
Therefore $M_{\mathbf{n}_i}$ contains a non-zero $w$ such that $Xw=0$ and the same arguments as above show that
$M$ is a weight module.

It remains to consider the case when $X$ acts injectively on $M$. Since $M$ is finite dimensional over
$\Bbbk$ and $X$ is $\Bbbk$-linear (as $\sigma$ is $\Bbbk$-linear), we have that $X$ acts bijectively on $M$. 
Then any $v\in M_{\mathbf{n}_i}$ is of the form $v=Xw$ for some $w\in M^{\mathbf{n}_{i-1}}$
(by convention, we have  $\mathbf{n}_{0}=\mathbf{n}_{k}$). Using \eqref{eq1}, we compute:
\begin{displaymath}
X\mathbf{n}_{i-1}w=\sigma(\mathbf{n}_{i-1})Xw\subset \mathbf{n}_iv=0.
\end{displaymath}
As $X$ acts injectively, we obtain $\mathbf{n}_{i-1}w=0$, which implies $w\in M_{\mathbf{n}_{i-1}}$.
Therefore $l_{i-1}\geq l_i$ and the fact that the graph $\Gamma$ is a cycle implies $l_1=l_2=\dots=l_k$. 
This means that  the action of $X$ preserves $\mathrm{wt}(M)$. Since the action of $R$ obviously preserves 
$\mathrm{wt}(M)$, it follows that $\mathrm{wt}(M)$ is an $A$-submodule and thus coincides with
$M$ by simplicity of the latter.

\section{Application to generalized Heisenberg algebras}\label{s6}

\subsection{General reduction}\label{s6.1}

In this section we consider a number of concrete applications of the above results to classifications
of simple weight modules over generalized Heisenberg algebra $\mathcal{H}(f)$ defined in the
introduction, where $f(h)\in\mathbb{C}[h]$. Note that the element $z$ is central in $\mathcal{H}(f)$
and hence, by Schur's lemma, it acts as a scalar on each simple $\mathcal{H}(f)$-module.

Fix $\dot{z}\in \mathbb{C}$ and consider the quotient $\mathcal{H}(f)_{\dot{z}}$ of $\mathcal{H}(f)$ by the
principal central ideal generated by $(z-\dot{z})$. Note that the algebra $\mathcal{H}(f)_{\dot{z}}$ is a wGWA for 
$R=\mathbb{C}[h]$, $t=h+\dot{z}$ and $\sigma:R\to R$ defined by $\sigma(h)=f(h)$. Clearly, $\sigma$ is essential.

For $\chi\in \mathbb{C}$ we denote by $\mathbf{m}_{\chi}$ the maximal ideal $(h-\chi)$ in $R$ and in this
way identify $\mathrm{Max}(R)$ with $\mathbb{C}$. For $\chi\in \mathbb{C}$ the ideal 
$\sigma(\mathbf{m}_{\chi})$ is generated by $f(h)-\chi$. Therefore, for $\eta\in \mathbb{C}$ we have
$\sigma(\mathbf{m}_{\chi})\subset \mathbf{m}_{\eta}$ if and only if $f(\eta)=\chi$. Conversely, 
$\mathbf{m}_{f(\chi)}$ is the unique maximal ideal such that $\sigma(\mathbf{m}_{f(\chi)})\subset \mathbf{m}_{\chi}$.
Note that all this works even if $f$ is constant.

We have $t=h+\dot{z}\in \mathbf{m}_{-\dot{z}}$ and $t\not\in \mathbf{m}_{\chi}$ for $\chi\neq -\dot{z}$.
We also have $\sigma(t)=f(h)+\dot{z}\in \mathbf{m}_{\chi}$ if and only if $f(\chi)=-\dot{z}$.

Denote by $\Phi$ the oriented graph with vertices $\mathbb{C}$ and arrows $f(\chi)\leftarrow \chi$
for all $\chi\in\mathbb{C}$. Theorem~\ref{thm23} reduces classification of simple weight 
$\mathcal{H}(f)_{\dot{z}}$-modules to description of the dynamics of the action of the transformation  
$\chi\mapsto f(\chi)$ of $\mathbb{C}$, that is oriented paths in the graph $\Phi$.

\subsection{Degree zero case}\label{s6.2}

If $f$ has degree zero, then $f$ is a  constant polynomial, say with value $\theta$. 
Assume first that $\theta\neq-\dot{z}$. Then from  Theorem~\ref{thm23} we have that
$\mathcal{H}(f)_{\dot{z}}$ has a unique infinite dimensional simple weight module, namely 
$N^{\downarrow}(\underline{\mathbf{m}})$, where 
\begin{displaymath} 
\underline{\mathbf{m}}=(\dots,\mathbf{m}_{\theta},\mathbf{m}_{\theta},\mathbf{m}_{\theta},\mathbf{m}_{-\dot{z}});
\end{displaymath}
and a family of one-dimensional simple weight modules with support $\mathbf{m}_{\theta}$
indexed by $c\in \mathbb{C}^*$. Each module from the latter family coincides with 
$R/(h-\theta)$ as $R$-modules, on the module indexed by $c$ the element $X$ acts as $c$ and
the element $Y$ acts as $\frac{\theta+\dot{z}}{c}$.

If $\theta=-\dot{z}$, then from  Theorem~\ref{thm23} we have that 
$\mathcal{H}(f)_{\dot{z}}$ has the one dimensional simple weight module which coincides with 
$R/(h-\theta)$ as an $R$-module and on which both $X$ and $Y$
act as zero and, additionally, exactly two families of non-isomorphic simple weight modules
indexed by $\mathbb{C}^*$, all modules having dimension one. All modules in both families coincide with 
$R/(h-\theta)$ as $R$-modules. On modules of the first family $X$ acts as $c\in \mathbb{C}^*$ 
and $Y$ as zero. On modules of the second family $Y$ acts as $c\in \mathbb{C}^*$ 
and $X$ as zero.

\subsection{Degree one case}\label{s6.3}

If $f$ has degree one and $f(h)\neq h$ (the latter means that $A$ is not commutative), 
then for every $\chi\in\mathbb{C}$ there is a unique $\eta\in\mathbb{C}$ such that
$f(\eta)=\chi$. Let $\alpha\in\mathbb{C}$ be the unique number such that $f(\alpha)=-\dot{z}$.
From  Theorem~\ref{thm23} it thus follows that simple weight modules over
$\mathcal{H}(f)_{\dot{z}}$ are: $N^{\downarrow}(\underline{\mathbf{m}})$ where
\begin{displaymath} 
\underline{\mathbf{m}}=(\dots,\mathbf{m}_{f(f(f(-\dot{z})))},\mathbf{m}_{f(f(-\dot{z}))},\mathbf{m}_{f(-\dot{z})},\mathbf{m}_{-\dot{z}});
\end{displaymath}
$N^{\uparrow}(\underline{\mathbf{m}})$ where
\begin{displaymath} 
\underline{\mathbf{m}}=(\mathbf{m}_{\alpha},\mathbf{m}_{\alpha_1},\mathbf{m}_{\alpha_2},\mathbf{m}_{\alpha_3},\dots)
\end{displaymath}
with $f(\alpha_i)=\alpha_{i-1}$ for $i=1,2,\dots$; and $N(\underline{\mathbf{m}})$ where
\begin{displaymath} 
\underline{\mathbf{m}}=(\dots,\mathbf{m}_{\beta_{-2}},\mathbf{m}_{\beta_{-1}},\mathbf{m}_{\beta_{0}},
\mathbf{m}_{\beta_{1}},\mathbf{m}_{\beta_{2}},\dots)
\end{displaymath}
with $f(\beta_i)=\beta_{i-1}$ and $\beta_i\neq -\dot{z}$ for $i\in\mathbb{Z}$.

Note that in this case $\mathcal{H}(f)_{\dot{z}}$ is a genuine GWA.

\subsection{The case of $h^n$ for $n>1$}\label{s6.4}

Assume now that $f(h)=h^n$ for some $n>1$. In this case we have one-dimensional $\mathcal{H}(f)_{\dot{z}}$-modules
with support $\mathbf{m}_{0}$. If $\dot{z}\neq 0$, we have one $\mathbb{C}^*$-indexed family of 
one-dimensional simple weight $\mathcal{H}(f)_{\dot{z}}$-modules with support $\{\mathbf{m}_{0}\}$
(defined similarly as in Subsection~\ref{s6.2}). If $\dot{z}= 0$, we have two $\mathbb{C}^*$-indexed families 
of  one-dimensional simple weight $\mathcal{H}(f)_{\dot{z}}$-modules with support $\{\mathbf{m}_{0}\}$, again
see Subsection~\ref{s6.2} for details. 

As $\mathbb{C}$ is algebraically closed, each $\mathbf{m}_{\chi}$ appears in some
$\underline{\mathbf{m}}\in {}_{\infty}Q_{\infty}$. The graph $\Phi$ has the connected component $\{0\}$
which is considered above. Another isolated component is the set of all roots of unity.
This is exactly the isolated component in which each path eventually converges to an oriented cycle,
in particular, it contains a lot of periodic elements in ${}_{\infty}Q_{\infty}$. Non-periodic
paths in this connected component give either double infinite string modules (if this path does not cross
$\dot{z}$) or left- or right-infinite string modules if the path crosses $\dot{z}$. Periodic paths give band modules
(again, one family of band modules with one dimensional weight spaces of the path does not
cross $\dot{z}$, and two families if the path crosses $\dot{z}$). band modules are finite dimensional.

Finally, we have the isolated component consisting of non-zero non roots of unity. Similarly to the above
this component gives double infinite string modules or  left- or right-infinite string modules.

As $t$ has degree one, bounded string modules do not exist.

\begin{example}\label{ex1257}
{\rm
Here is an explicit example of a simple weight module which has  both non-trivial finite dimensional
and infinite dimensional weight spaces. Take $f(h)=h^2$. Define $\theta_i=1$ for all $i\leq 0$ and
also define $\theta_j=\exp\big(\frac{2\pi\mathtt{i}}{2^j}\big)$ for $j\geq 1$ 
(here $\mathtt{i}$ denotes the imaginary unit). Note that $\theta_j^2=\theta_{j-1}$ for all $j\in\mathbb{Z}$.
Let $\dot z\in \mathbb{C}$ be such that $\dot z+\theta _i\ne0$ for all $i$. Finally, set 
\begin{displaymath}
\underline{\mathbf{m}}=(..., \mathbf{m}_{\theta_{-1}}, 
\mathbf{m}_{\theta_{0}}, \mathbf{m}_{\theta_{1}}, \mathbf{m}_{\theta_{2}}, ...). 
\end{displaymath}
Then the $\mathcal{H}(f)_{\dot{z}}$-module $N(\underbar m)$ is simple, its weight space 
$N(\underbar m)_{\mathbf{m}_1}$ is infinite dimensional while all other nonzero weight spaces have dimension one.
}
\end{example}

 

\vspace{2mm}

\noindent
R.L.: Department of Mathematics, Soochow university, Suzhou 215006,
Jiangsu, P. R. China; e-mail: {\tt rencail\symbol{64}amss.ac.cn}
\vspace{3mm}

\noindent
V.M.: Department of Mathematics, Uppsala University,
Box 480, SE-751 06, Uppsala, Sweden; e-mail: {\tt mazor\symbol{64}math.uu.se}
\vspace{3mm}

\noindent K.Z.: Department of Mathematics, Wilfrid Laurier
University, Waterloo, Ontario, N2L 3C5, Canada; and College of Mathematics and
Information Science, Hebei Normal (Teachers) University, Shijiazhuang 050016,
Hebei, P. R. China. e-mail:  {\tt kzhao\symbol{64}wlu.ca}

\end{document}